\documentclass{article}

\usepackage{graphicx} 
\usepackage{cite}
\usepackage{authblk}
\usepackage{amsthm}
\usepackage{amsmath}
\usepackage{amssymb}
\usepackage{esint}
\usepackage[left=3cm, right=3cm, lines=45, top=1.0in, bottom=1.0in]{geometry}
\usepackage[colorlinks=true]{hyperref}

\title{\textbf{Lower Bound of Nodal Sets in Elliptic Homogenization and Functions with Strong Maximum Principle}}
\author{Jiahuan Li, Zhichen Ying}
\date{2025}

\begin{document}

\maketitle
\newtheorem{theorem}{Theorem}[section] 
\newtheorem{definition}[theorem]{Definition} 
\newtheorem{lemma}[theorem]{Lemma}
\newtheorem{corollary}[theorem]{Corollary}
\newtheorem{example}[theorem]{Example}
\newtheorem{proposition}[theorem]{Proposition}
\newtheorem{conjecture}[theorem]{Conjecture}
\newtheorem{remark}[theorem]{Remark}

\newcommand{\keywords}[1]{\par\quad\textbf{Keywords:} #1}
\newcommand{\classification}[1]{\par\quad\textbf{AMS Classification:} #1}
\renewcommand{\thefootnote}{}

\begin{abstract}
    In this paper, we first prove a uniform lower bound of nodal volume in elliptic homogenization setting. This lower bound is far from optimal. But, we can prove a constant lower bound in dimension two. Motivated by the proof, we extend this result to more general settings. To be more specific, we prove that the nodal volume has a constant lower bound for all continuous functions with strong maximum principle. Our results work for general functions beyond solutions to elliptic PDEs.
\end{abstract}


\numberwithin{equation}{section}

\section{Introduction}

The analysis of nodal sets in elliptic PDEs is a challenging problem. In the work of Lin \cite{MR1090434}, it has been proved that the nodal sets of elliptic PDEs have Hausdorff measure less than or equal to $n-1$. Then, it is natural to ask whether the nodal volume has upper bound and lower bound locally and globally. When considering the Laplace eigenfunction on Riemannian manifold, Yau made the following conjecture:
\begin{conjecture}[\cite{MR645728}]\label{Yau}
    Let $(M^{n},g)$ be a closed and compact manifold and $u$ be a nonconstant eigenfunction that satisfies $\Delta_{g} u+\lambda u=0$ with $\lambda>0$. Then
    \begin{equation}
        C^{-1}(M,g)\sqrt{\lambda}\leq \mathcal{H}^{n-1}(Z(u))\leq  C(M,g)\sqrt{\lambda}.
    \end{equation}
\end{conjecture}
For the lower bound, it can be reduced to Nadirashvili's conjecture. It was conjectured in \cite{nadirashvilli1997geometry} that, if a harmonic function vanishes at origin, then the Hausdorff measure of nodal sets has a constant lower bound in unit ball. This conjecture has been solved by Logunov in \cite{MR3739232}. Another version of Nadirashvili's conjecture was solved by the authors and Wang in \cite{li2025nadirashviliconjectureellipticpdes}. Actually, in \cite{MR4814921}, Logunov, Lakshmi Priya and Sartori proved a much stronger version. They proved an almost sharp lower bound of nodal volume of elliptic PDEs. We will recall the main theorem here:
\begin{theorem} \label{almost sharp lower bound}
    Let $B\subset\mathbb{R}^{n}$ be a unit ball and $n\geq 3$. For every $\varepsilon>0$, there exists a constant $c$ depending on dimension $n$ and $\varepsilon$ such that for every harmonic function $u : 4B\rightarrow \mathbb{R}$ with $u(0)=0$, we have
    \begin{equation}
        \mathcal{H}^{n-1}(\{u=0\}\cap 2B)\geq cN_u(0,\frac{1}{2})^{1-\varepsilon}
    \end{equation}
    Here $N_u(x,r)$ is the doubling index introduced later.
\end{theorem}
Indeed, this lower bound works for more genenral elliptic PDEs with $C^1$ coefficients. In their work, they made the following conjecture:
\begin{conjecture}
     Let $B\subset\mathbb{R}^{n}$ be a unit ball and $n\geq 3$. There exists a constant $c>0$ only depending on dimension $n$ such that for every harmonic function $u:4B\rightarrow \mathbb{R}$ with $u(0)=0$, we have
    \begin{equation}
        \mathcal{H}^{n-1}(\{u=0\}\cap 2B)\geq cN_u(0,\frac{1}{2})
    \end{equation}
\end{conjecture}
However, this conjecture remains open. Besides the estimate of lower bound, there are also many results regarding the upper bound of nodal sets for elliptic PDEs and especially the Laplace eigenfunction. In \cite{MR3739231}, Logunov proved a polynomial upper bound of nodal sets of Laplace eigenfunction on smooth manifold. Before that, Donnelly and Fefferman proved the Yau's conjecture completely for Laplace eigenfunction on real-analytic manifold in \cite{MR943927} and \cite{MR1035413}. For more general settings, Lin proved the sharp upper bound of nodal volume for elliptic PDEs with real-analytic coefficients in \cite{MR1090434}. Hardt and Simon proved a exponential upper bound for elliptic PDEs with Lipschitz coefficient in \cite{MR1010169}. Later, Han and Lin gave a different proof in \cite{Hanlin94}. Han also used Schauder estimate to stratify the nodal sets in \cite{MR1794573}. In \cite{MR3688031}, Naber and Valtorta gave a upper bound estimate for elliptic PDEs with Lipschitz coefficients. There are also many results concerning solutions to other equations including parabolic equations, bi-harmonic functions and so on. The reader may refer \cite{MR3928756}\cite{huang2024nodal}\cite{huang2023volume} and other related references.

An important theory in elliptic PDE is the homogenization theory. There is a small parameter $\varepsilon$ in the setting. The coefficients are highly oscalliting which makes previous results fail to apply. In \cite{MR3952693}, Lin and Shen first proved a uniform upper bound of nodal volume in elliptic homogenization. They used the compactness method to proved the doubling inequality. Later, Kenig, Zhu and Zhuge gave an explicit formula of upper bound in \cite{MR4387203}. Also, this upper bound is far from optimal. But they used conformal mapping to prove a refined upper bound in dimension two. In \cite{MR4764741}, Lin and Shen also studied the upper bound of critical sets in elliptic homogenization. They used the method of Naber and Valtorta in their work \cite{MR3688031}. They also proved the upper bound of critical sets in dimension two with a different method in \cite{MR4669024}.

It is a natural question to ask whether we can prove a similar lower bound of nodal volume in homogenization. We will begin with basic settings. In the first part of our paper, we will consider the following elliptic PDE in $B(0,1)\subset \mathbb{R}^{n}$:
\begin{equation}
    \mathcal{L}_{\varepsilon}(u_{\varepsilon})=\mathrm{div}(A(x/\varepsilon)\nabla u_{\varepsilon})=0
\end{equation}
Here, the matrix $A$ satisfies the following three properties:

(1)The matrix $A$ is unifrom elliptic. It means for $\forall \xi \in \mathbb{R}^n$, there exists a constant $\lambda>0$ such that $\lambda^{-1}|\xi|^2\leq a_{ij}\xi_i\xi_j\leq \lambda|\xi|^2$

(2)The matrix $A$ is 1-periodic. It means $A(y+z)=A(y), \forall y \in \mathbb{R}^n,z\in  \mathbb{Z}^n$;

(3)The matrix $A$ is $C^1$. 

Denote all matrix satisfying above three conditions by $\mathcal{A}$. We may assume $A$ is smooth enough first, and we believe that we can consider this problem with less regular coeffecients. When considering the possible outcome, we think there are three different possible results:

(1)The lower bound is independent of $\varepsilon$ and $N$;

(2)The lower bound is independent of $\varepsilon$ but depends on $N$;

(3)The lower bound is independent of $N$ but depends on $\varepsilon$;

We do not expect to prove the first version of results. So in this paper, we try to prove a lower bound depending on the doubling index when $\varepsilon>0$ is small. Also, we hope we can get a explicit form of the lower bound for all $\varepsilon\in (0,1)$. Before that, we will first introduce the definition of doubling index.

\begin{definition}
    Let $u$ be a solution to $\mathrm{div}(A\nabla u)=0$ in $B(0,1)$, we can define the doubling index of $u$ as 
    \begin{equation}
        N(x,r):=\log_{2}\frac{\int_{B(x,r)}|u|^{2}}{\int_{B(x,r/2)}|u|^{2}}
    \end{equation}
    with $B(x,r)\subset B(0,1)$. Sometimes, we will write $N(B)=N(x,r)$ if $B=B(x,r)$. Also we will simply write $N(r)$ if $x$ is fixed or $x$ is the origin. If we want to specify the choice of function, we will write $N_{u}(x,r)$.
\end{definition}

Our first main result is the following:
\begin{theorem} \label{uniform lower bound}
    Assume $A\in \mathcal{A}$. Let $u_{\varepsilon}\in H^{1}(B(0,1))$ be a solution of $\mathcal{L}_{\varepsilon}(u_{\varepsilon})=0$ in $B(0,1)\subset \mathbb{R}^n,n\geq 3$, where $0<\varepsilon<1$. Let $N_{0}=N_{u_{\varepsilon}}(0,1)$. Then there exists a constant $\varepsilon_{0}$ depending on $A,n,N_{0}$ such that if $\ 0<\varepsilon<\varepsilon_{0}$, the following estimate holds:
    \begin{equation}
        \mathcal{H}^{n-1}(\{u_{\varepsilon}=0\}\cap B(0,1))\geq \frac{C}{N_{0}^{n-1}}
    \end{equation}
    whenever $u_{\varepsilon}(0)=0$ and $C$ only depends on $A$ and $n$.
\end{theorem}
Combining Theorem \ref{almost sharp lower bound}, we conclude the following uniform lower bound estimate of homogenization equations:
\begin{theorem}
    Assume $A\in \mathcal{A}$. Let $u_{\varepsilon}\in H^{1}(B(0,1))$ be a solution of $\mathcal{L}_{\varepsilon}(u_{\varepsilon})=0$ in $B(0,1)\subset \mathbb{R}^n,n\geq 3$, where $0<\varepsilon<1$. Let $N_{0}=N_{u_{\varepsilon}}(0,1)$. Then 
    \begin{equation}
        \mathcal{H}^{n-1}(\{u_{\varepsilon}=0\}\cap B(0,1))\geq C\left( N_0 \right)
    \end{equation}
    whenever $u_{\varepsilon}(0)=0$ and $C\left( N_0 \right)$ only depends on $A,n,N_0$
\end{theorem}
In dimension two, we use a different approach. We will only use blow up argument and maximum principle to derive a constant lower bound:
\begin{theorem}\label{thm1.5}
    Assume $A\in \mathcal{A}$. Let $u_{\varepsilon}\in H^{1}(B(0,1))$ be a solution of $\mathcal{L}_{\varepsilon}(u_{\varepsilon})=0$ in $B(0,1)\subset \mathbb{R}^2$. Then there exsits a constant $C_{0}>0$ depending only on $A$ such that the following estimate holds
    \begin{equation}
        \mathcal{H}^{1}(\{u_{\varepsilon}=0\}\cap B(0,1))\geq C_{0}
    \end{equation}
    whenever $u_{\varepsilon}(0)=0$
\end{theorem}

Motivated by the proof of above theorem, we try to give a much simpler proof of Nadirashvili's conjecture in dimension two. Indeed, we can prove Nadirashvili's conjecture in a more general setting. We only assume the function satisfies strong maximum principle and is continuous. Here, we recall the content of Nadirashvili's conjecture for harmonic functions:
\begin{theorem}[Nadirashvili's conjecture]
    Assume that $u$ is a harmonic function in $\mathbb{R}^n$ and $u(0)=0$. Then we have the estimate
    \begin{equation}
        \mathcal{H}^{n-1}(\{u=0\}\cap B(0,1))\geq c(n)
    \end{equation}
    for some constant $c(n)>0$ only depending on dimension.
\end{theorem}
Before stating the last main result, we need the following definition:
\begin{definition}
    Let $B_1\subseteq \mathbb{R}^n$
    \begin{equation}
    SMP\left( B_1 \right) :=\left\{ u\in C^0\left( \overline{B_1} \right) :\text{for    }\forall \varOmega \subseteq B_1,\varOmega \text{ open},u|_{\varOmega}
\text{ cannot take maximum or minimum } in\,\,\varOmega \right\}   
    \end{equation}
    And we say $u$ is an SMP function if $u\in SMP\left( B_1 \right) $.
\end{definition}
Then we will state our last main result in this paper as the following:
\begin{theorem} \label{main results}
   We have the following equality holds:
    \begin{equation}
        \underset{\begin{array}{c}
	u\in SMP\left( B_1 \right)
\end{array}}{inf}\mathcal{S}^1( \left\{ u=u(0) \right\} \cap \overline{B_1} )= 2
    \end{equation}
\end{theorem}
In fact, we can pose weaker condition to $u$ as below:
\begin{definition}
        Let
    \begin{equation}
    WSMP\left( B_1 \right) :=\left\{ u\in C^0\left( \overline{B_1} \right) :\text{for    }\forall 0<r\leq 1,u|_{B_r}
 \text{ cannot take maximum or minimum } in\,\,B_r \right\}   
    \end{equation}
    We say $u$ is a WSMP function if $u\in WSMP\left( B_1 \right) $.
\end{definition}
By a similar argument, we can derive the following theorem:
\begin{theorem}
    We have the following equality holds:
    \begin{equation}
        \underset{\begin{array}{c}
	u\in WSMP\left( B_1 \right)\\
\end{array}}{inf}\mathcal{S}^1( \left\{ u=u(0) \right\} \cap \overline{B_1} )= 2
    \end{equation}
\end{theorem}

We mention again that this theorem actually works for a wide range of solutions of PDEs not only solutions to elliptic PDEs. This results shows that the lower bound in dimension two is a conseqence of maximum principle rather than the structure of PDE. The lower bound $2$ is sharp. In fact, if we assume the set $\left\{ u=0 \right\} \cap B_1$ is countably rectifiable, then this theorem is easy to prove by coarea formula. Here we do not assume it is countably rectifiable. Thus, we need a new analytical method to prove the lower bound of the nodal set.

\textbf{Notations.} We use the following notations in this paper.
\begin{itemize}
    \item We denote positive constants by $C,C_1,\cdots$. We may write $C(a_{1},a_{2},\cdots)$ to highlight its dependence on parameters $a_{1},a_{2},\cdots$. The value of $C$ may vary from line to line.
    \item We denote the ball centered at $x$ with radius $r$ by $B(x,r)=\{y\in \mathbb{R}^{n}:|x-y|\leq r\}$. We will write $B_r$ if $x$ is fixed or it is $0$.
    \item We denote the $n$-dimensional Hausdorff measure by $\mathcal{H}^{n}$.
    \item We denote the $n$-dimensional Spherical measure by $\mathcal{S}^{n}$.
    \item For a given function $u$, we sometimes denote its nodal set by $Z(u)=\{x:u(x)=0\}$.
\end{itemize}

\indent \textbf{Acknowledgements.} Jiahuan Li was supported by National Natural Science Foundation of China [grant number 12141105]. The authors would like to thank anyone who read the draft of this paper and give useful comments.

\section{Uniform Lower Bound for Elliptic Homogenization}

\subsection{Properties from Homogenization Theory}

In this subsection, we will first collect several important properties from homogenization theory. Most of these results can be found in \cite{MR4764741} and \cite{MR4387203}.

\begin{theorem} \label{doubling}
    Assume $A\in \mathcal{A}$. Let $u_{\varepsilon}\in H^{1}(B(0,2))$ be a solution of $\mathcal{L}_{\varepsilon}(u_{\varepsilon})=0$ in $B(0,2)\subset\mathbb{R}^n,n\geq 3$. If $u_{\varepsilon}$ satisfies the following duobling inequality
    \begin{equation}
        \int_{B_{2}}u_{\varepsilon}^{2}\leq N\int_{B_{1}}u_{\varepsilon}^{2}
    \end{equation}
    then for any $x\in B_{1/3}$ and any $B_{2r}(x)\subset B_{2}$, we have
    \begin{equation}
        \int_{B_{2r}(x)}u_{\varepsilon}^{2}\leq \mathrm{exp}(\mathrm{exp}(CN^{\frac{2}{\beta-\frac{3}{4}}}))\int_{B_{r}(x)}u_{\varepsilon}^{2}
    \end{equation}
    Here $C$ only depends on $A,n$ and $\beta$ is any number in the interval $(\frac{3}{4},1)$.
\end{theorem}

\begin{theorem} \label{Approximation}
    Assume $A\in \mathcal{A}$. Let $u_{\epsilon}\in H^{1}(B(0,2r))$ be a solution of $\mathcal{L}_{\varepsilon}(u_{\varepsilon})=0$ in $B(0,2r)$. Suppose $r>C\sqrt{N}\varepsilon$ for some large $C$ and $u_{\varepsilon}$ satisfies the doubling inequality
    \begin{equation}
        \int_{B_{2r}}u_{\varepsilon}^{2}\leq N\int_{B_{r}}u_{\varepsilon}^{2}
    \end{equation}
    Then there exists a harmonic function $u_{0}$ in $B_{\frac{7}{4}r}$ such that
    \begin{equation}
        ||u_{\varepsilon}-u_{0}||_{L^{\infty}(B_{\frac{3}{2}r})}\leq \frac{C\varepsilon}{r}(\fint_{B_{2r}}u_{\varepsilon}^{2})^{\frac{1}{2}}
    \end{equation}
    and 
    \begin{equation}
        \int_{B_{r}}u_{0}^{2}\leq 16N^{2}\int_{B_{r/2}}u_{0}^{2}
    \end{equation}
    where $C$ only depends on $A,n$.
\end{theorem}

\subsection{Approximation and Volume Estimate}

In this subsection, we will use approximation to derive the desired lower bound for a fixed doubling index. But we still need to mention here that this lower bound is far from optimal. We will use harmonic approximation and the methods in \cite{MR4814921}.

We first illustrate the following Harnack-type inequality:
\begin{lemma} \label{Harnack}
    Let $B=B(x,r)\subset \mathbb{R}^{n}$ be any ball. Then for any non-constant harmonic function $u$ defined in $2B$ satisfies $u(x)\geq 0$, we have
    \begin{equation}
        \sup_{\frac{4}{3}B}u\geq C\sup_{B}|u|
    \end{equation}
    for some $C$ only depends on $n$.
\end{lemma}

The proof can be found in \cite{MR4814921}. Next we will prove the following lemma:

\begin{lemma} \label{Main Lemma}
    Suppose $u_{0}$ is harmonic function in $B(0,4),u_{0}(0)=0$ and $u$ is another continuous function. Assume $u_{0}$ is not constant and satisfies the following doubling inequality
    \begin{equation}
        \log_{2}\frac{\sup_{B(0,4)}|u_{0}|}{\sup_{B(0,2)}|u_{0}|}\leq N
    \end{equation}
    Then there exists a constant $\varepsilon>0$ depending on $n$ and a constant $C$ depending on $N,n$ such that if 
    \begin{equation}
        \sup_{B(0,4)}|u-u_{0}|\leq \varepsilon\sup_{B(0,1)}|u_{0}|
    \end{equation}
    then we have 
    \begin{equation}
        \mathcal{H}^{n-1}(\{u=0\}\cap B(0,4))\geq C
    \end{equation}
\end{lemma}

\begin{proof}
    Denote $2=r_{0}<r_{1}<\cdots<r_{k}=3$ and $r_{i}-r_{i-1}=\frac{1}{k}$. Without loss of generality, we may assume $\sup_{B(0,1)}|u_{0}|=1$. Let $M_{i}=\sup_{B_{r_{i}}}u_{0},m_{i}=\inf_{B_{r_{i}}}u_{0}$. By Lemma \ref{Harnack}, we know that there exists a constant $C_{1}$ such that
    \begin{equation}
        M_{i}\geq C_{1},m_{i}\leq -C_{1}
    \end{equation}
    Given a constant $S$, we say the index $i$ is $S^{+}$-good if 
    \begin{equation}
        M_{i+1}\leq SM_{i}
    \end{equation}
    We say the index $i$ is $S^{-}$-good if 
    \begin{equation}
        -m_{i+1}\leq -Sm_{i}
    \end{equation}
    Let $k^{+}=\# \{i:i \text{ is } S^{+} \text{-good}\}$ and $k^{-}=\# \{i:i \text{ is } S^{-} \text{-good}\}$. Then by definition, we know that
    \begin{equation}
         \log \frac{M_{k}}{M_{0}}=\sum_{i \text{ is } S^{+}\text{-good}}\log\frac{M_{i+1}}{M_{i}}+\sum_{i \text{ is not } S^{+}\text{-good}}\log\frac{M_{i+1}}{M_{i}}\geq \log S^{k-k^{+}}
    \end{equation}
    At the same time, we have
    \begin{equation}
        \frac{M_{k}}{M_{0}}=\frac{\sup_{B(0,3)}u_{0}}{\sup_{B(0,2)}u_{0}}\leq \frac{\sup_{B(0,4)}|u_{0}|}{C_{1}}\leq 2^{CN}
    \end{equation}
    So we get
    \begin{equation}
        S^{k-k^{+}}\leq 2^{CN}
    \end{equation}
    As a result, we obtain a lower bound of $k^{+}$ as the following
    \begin{equation}
        k^{+}\geq k-CN
    \end{equation}
    Similarily, we can get the same lower bound for $k^{-}$. Then we take $k=3CN$. So we have
    \begin{equation}
        k^{+}\geq 2CN\geq \frac{2}{3}k,\quad k^{-}\geq 2CN\geq \frac{2}{3}k
    \end{equation}
    So we know that there are at least $\frac{k}{3}$ indices that are both $S^{+}$-good and $S^{-}$-good. For any $i$ that is both $S^{+}$-good and $S^{-}$-good, we assume $u_{0}(x_{1})=\sup_{B_{i}}u_{0}$. By standard gradient estimate, we know that
    \begin{equation}
        \sup_{B(x_{i},1/4k)}|\nabla u_{0}|\leq Ck\sup_{B(x_{i},1/3k)}| u_{0}|\leq Ck\sup_{B(x_{i},1/2k)}u_{0}
    \end{equation}
    where we used Lemma \ref{Harnack} in the last inequality. At the same time, we have
    \begin{equation}
        Ck\sup_{B(x_{i},1/2k)}u_{0}\leq CkM_{i+1}\leq CkSM_{i}=CNSM_{i}
    \end{equation}
    by our choice of $k$. By mean-value theorem, we know that in the ball $B(x_{i},\frac{C}{N})$, the following estimate holds:
    \begin{equation}
        u_{0}(x)\geq \frac{1}{2}M_{i}\geq C
    \end{equation}
    We can also apply this argument to $u_{0}(x_{2})=\inf_{B_{i}}u_{0}$. We know that in the ball $B(x_{2},\frac{C}{N})$, the following estimate holds:
    \begin{equation}
        u_{0}(x)\leq -\frac{1}{2}m_{i}\leq -C
    \end{equation}
    So if $\varepsilon<\frac{1}{2}C$, we know that
    \begin{equation}
        u(x)>0\,\,\,\, \text{in } B(x_{1},\frac{C}{N}) \text{ and } u(x)<0\,\,\,\, \text{in } B(x_{2},\frac{C}{N})
    \end{equation}
    By standard nodal sets estimate, we know that
    \begin{equation}
        \mathcal{H}^{n-1}(\{u=0\}\cap B(0,4))\geq \frac{C}{N^{n-1}}
    \end{equation}
    
\end{proof}

With this lemma, we can prove our final results:

\begin{proof}[Proof of Theorem \ref{uniform lower bound}]
    Indeed, we only need to verify conditions in Lemma \ref{Main Lemma}. The second condition is satisfied by Theorem \ref{Approximation} after scaling and doubling inequality. We just need to verify the first condition and the doubling inequality. Let us start with the following estimate
    \begin{align}
        ||u_{\varepsilon}||_{L^{2}(B_{3/16})}&\leq ||u_{\varepsilon}-u_{0}||_{L^{2}(B_{3/16})}+||u_{0}||_{L^{2}(B_{3/16})}\\
        &\leq C (||u_{\varepsilon}-u_{0}||_{L^{\infty}(B_{3/16})}+||u_{0}||_{L^{\infty}(B_{3/16})})\\
        &\leq C(N)\varepsilon ||u_{\varepsilon}||_{L^{2}(B_{3/16})}+||u_{0}||_{L^{\infty}(B_{3/16})}
    \end{align}
    So if $\varepsilon<\frac{1}{2C(N)}$, we know that
    \begin{equation}
        ||u_{\varepsilon}||_{L^{2}(B_{3/16})}\leq C(N)||u_{0}||_{L^{\infty}(B_{3/16})}
    \end{equation}
    Then we get
    \begin{equation}
        ||u_{\varepsilon}-u_{0}||_{L^{\infty}(B_{3/8})}\leq C(N)\varepsilon ||u_{0}||_{L^\infty(B_{3/16})}
    \end{equation}
    This verifies the first condition. If $\varepsilon<\varepsilon_{0}(N)$, by Theorem \ref{Approximation}, we know that
    \begin{equation}
        \int_{B_{1}}u_{0}^{2}\leq 16N^{2}\int_{B_{1/2}}u_{0}^{2}
    \end{equation}
    By $L^{\infty}$ approximation and standard argument, we know that
    \begin{equation}
        \log \frac{\sup_{B_{3/8}}|u_{0}|}{\sup_{B_{3/16}}|u_{0}|}\leq C(N)
    \end{equation}
    Finally, by Lemma \ref{Main Lemma}, we have the Theorem \ref{uniform lower bound} holds.
    
\end{proof}

\section{Special Case in Two Dimension}

In this section, we will try to prove the constant lower bound in dimension two. We will prove a stronger conclusion. And we will use methods from elliptic PDEs. All the functions in this sections are defined in $\mathbb{R}^2$. First, we need the following constant lower bound for elliptic PDEs. This lemma does not imply Theorem \ref{thm1.5} due to the presence of $\varepsilon$.

\begin{lemma}[\cite{MR4814921}] \label{constant lower bound}
Assume A is uniform elliptic and $C^1$. Let $u\in H^1\left(Q_2\right) $ be a solution of $\mathrm{div}\left( A\nabla u \right) =0$ in $Q_2$. Then there exists a constant $C$ depending only on $A$ such that if $$\left\{ u=0 \right\} \cap Q_{1}\ne \emptyset$$
then$$\mathcal{H}^1( \left\{ u=0 \right\} \cap Q_2 )\geq C$$
where $Q_r=\left\{ \left( x_1,x_2 \right) ||x_i|\leq r,i=1,2 \right\}$.
\end{lemma}

\begin{proof}[Proof of Theorem \ref{thm1.5}]

We will also consider the following two cases.

Case 1: If $\varepsilon>1/10000$, then we will use the blow-up argument directly to get the lower bound:
    \begin{equation}
        \mathcal{H}^{n-1}(\{u_{\varepsilon}=0\}\cap B(0,1))\geq C\varepsilon^{n-1}\geq C
    \end{equation}
    for some $C$ only depending on $A$.

Case 2: If $0<\varepsilon\leq 1/10000$. WLOG we can assume that $k:=1/\varepsilon$ is an integer. We will consider the equation at a unified scale, and we consider the equation as
 \begin{equation}
        \mathrm{div}\left( A\nabla u \right) =0 \ in \ Q_k
    \end{equation}
We can assume that $u$ is not identically zero, otherwise the original theorem would obviously hold. We assume that
\begin{equation}
M_s:=\underset{Q_s}{\sup} u,\ m_s:=\underset{Q_s}{\inf} u\ \forall s>0
\end{equation}
By the maximum principle, we know that
\begin{equation}
M_s>0,\ m_s<0\ \forall s>0.
\end{equation}
Hence by the continuity of $u$, we know that for any $s>0$, there exists $x_s\in \partial Q_s$ such that $u\left( x_s \right) =0$. Therefore, for any positive integer $t\leq k$, we can find a closed cube $Q^t$ whose side length is 2, such that the center of $Q^t$ is an integer point and $\left\{ u=0 \right\} \cap \frac{1}{2}Q^t\ne \emptyset$.
By Lemma \ref{constant lower bound} and the periodicity of $A$, we know there exists a constant $C$ depending only on $A$ such that
\begin{equation}
\mathcal{H}^1( \left\{ u=0 \right\} \cap Q^t ) \geqslant C
\end{equation}
This constant is independent of $x$ by periodicity. We define the number of intersections between cubes as follows
\begin{equation}
b_j:=\#\left\{ i|Q^i\cap Q^j\ne \emptyset \right\} 
\end{equation}
Obviously, there exists an universal constant $C_1$ such that 
\begin{equation}
b_j\leq C_1, \forall 1\leq j\leq k
\end{equation}
So we can find at least $k/C_1$ non-intersecting cubes from $\left\{ Q^t \right\} _{t=1}^{k}$. So there exists a constant $C$ only depending on $A$ such that
\begin{equation}
\mathcal{H}^1(\left\{ u=0 \right\} \cap Q_k ) \geq Ck
\end{equation}
\end{proof}

In fact, We can prove a more general conclusion:
\begin{theorem}  \label{general conclusion}
We assume that $u$ is a solution to the equation $\mathrm{div}\left( A\nabla u \right) =0 $ in $B_1$ and the matrix satisfies the following conditions:
\begin{equation}
\begin{cases}
	\lambda ^{-1}|\xi |^2\leq a_{ij}\xi _i\xi _j\leq \lambda |\xi |^2 \,\,\,\,\forall \xi \in \mathbb{R}^2 \\ a_{ij}\in C^1\left( B_1 \right)\\
\end{cases}
\end{equation}
for some $\lambda>0$ .Then there exists a constant $C$ depending only on $\lambda $ such that if $u(0)=0$, then
\begin{equation}
    \mathcal{H}^1(\left\{ u=0 \right\} \cap B_1 ) \geq C
\end{equation}
\end{theorem}

This theorem implies that the lower bound is independent of Lipschitz constant of coefficients. To prove this theorem, we need two classic results and two important lemmas:

\begin{theorem}[Harnack inequality]
Assume that $A$ and $u$ satisfies the assumption in Theorem \ref{general conclusion}. There exists a constant C depending only on $\lambda$ such that if $u\geq 0 \ in\,\,B_1$ then
\begin{equation}
\underset{B_{1/2}}{\sup}u\leq C\underset{B_{1/2}}{\inf}u
\end{equation}
\end{theorem}
Then we prove the following Harnack-type inequality:

\begin{lemma}\label{Harnack type}
Assume $A$ and $u$ satisfy the assumption in Theorem \ref{Harnack}, then there exists a constant $C$ denpending only on $\lambda$ such that if $u(0)\geq0$, then
\begin{equation}
    \underset{B_{1/2}}{\sup}u\geq C\underset{B_{2/5}}{\sup}|u|
\end{equation}
\end{lemma}
\begin{proof}
Let $v:=\underset{B_{1/2}}{\sup}u-u$, then $v\geq 0$ in $B_{1/2}$. By Harnack inequality, we know that
\begin{equation}
    \underset{B_{2/5}}{\sup}v\leq C(\lambda)\underset{B_{2/5}}{\inf}v
\end{equation}
Hence, we have
\begin{equation}
    C\left( \lambda \right) \underset{B_{2/5}}{\sup}u-\underset{B_{2/5}}{\inf}u\leq \left( C\left( \lambda \right) -1 \right) \underset{B_{1/2}}{\sup}u
\end{equation}
If $\underset{B_{2/5}}{\sup}|u|=-\underset{B_{2/5}}{\inf}u$. By the fact that $u(0)\geq 0$, we know that $\underset{B_{2/5}}{\sup}u\geq0$. So we can get the following inequality
\begin{equation}
    \underset{B_{2/5}}{\sup}|u|\leq C\left( \lambda \right) \underset{B_{2/5}}{\sup}u-\underset{B_{2/5}}{\inf}u\leq \left( C\left( \lambda \right) -1 \right) \underset{B_{1/2}}{\sup}u
\end{equation}
If $\underset{B_{2/5}}{\sup}|u|=\underset{B_{2/5}}{\sup}u$, then
\begin{equation}
    \underset{B_{2/5}}{\sup}|u|=\underset{B_{2/5}}{\sup}u\leq\underset{B_{1/2}}{\sup}u
\end{equation}
This completes our proof.

\end{proof}

For convenience, in the following argument, we will use the following notation:
\begin{definition}
    We say a symmetric matrix $A=(a_{ij})$ belongs to ${A} \left( \lambda ,M \right) $, if 
    \begin{equation}
        \begin{cases}
	\lambda ^{-1}|\xi |^2\leq a_{ij}\xi _i\xi _j\leq \lambda |\xi |^2 \,\,\,\,\forall \xi \in \mathbb{R}^2 \\
	Lip\left( a_{ij} \right) \leq M\\
\end{cases}
    \end{equation}
\end{definition}
We need the following classical Schauder estimate:
\begin{theorem}[Schauder estimate]
If $A\in {A} \left( \lambda ,M \right) $ and $div\left( A\nabla u \right) =0\ in\ B_1$. Then there exists constants $C$ and $\alpha$ depending only on $\lambda,M$ such that
 \begin{equation}
    \left\| u \right\| _{C^{1,\alpha}\left( B_{3/4} \right)}\leq C\underset{B_{4/5}}{\sup}|u|
 \end{equation}   
\end{theorem}

Then we will prove the following doubling inequality:
\begin{lemma} \label{doubling lower}
Assume $A\in {A} \left( \lambda ,M \right) $, $\mathrm{div}\left( A\nabla u \right) =0  \ in\ B_1$ ,$u(0)=0$ and $u$ is not zero. There exists a constant $C>1$ depending only on $\lambda,M$ such that 
\begin{equation}
    \frac{\underset{B_{4/5}}{\sup}|u|}{\underset{B_{2/5}}{\sup}|u|}\geq C
\end{equation}
\end{lemma}

\begin{proof}
We argue by contradiction. If this lemma does not holds, then we can find $\left\{ A_k \right\} _{k=1}^{\infty}\subset A\left( \lambda ,M \right) , \left\{ u_k \right\} _{k=1}^{\infty}$ such that $\mathrm{div}\left( A_k\nabla u_k \right) =0\ in\ B_1, u_k\left( 0 \right) =0$, $u_k$ is not identically zero and
\begin{equation}
    \frac{\small{\underset{B_{4/5}}{\sup}|u_k|}}{\small{\underset{B_{2/5}}{\sup}|u_k|}}\leq 1+\small{\frac{1}{k}}
\end{equation}
By scaling, we can assume $\underset{B_{2/5}}{\sup}|u_k|=1$, hence $\underset{B_{4/5}}{\sup}|u_k|\leq 1+\frac{1}{k}$. By Schauder estimate and taking subsequence, we can assume that there exists $u^*\in C^{1,\alpha}\left( B_{3/4} \right)$ such that
\begin{equation}
    \underset{k\rightarrow \infty}{\lim}\left\| u_k-u^* \right\| _{C^{1,\alpha}\left( B_{3/4} \right)}=0
\end{equation}
And by taking subsequence, we also can assume that there exists $A^*\in A\left( \lambda ,M \right)$ such that 
\begin{equation}
    \underset{k\rightarrow \infty}{\lim}\left\| A_k-A^* \right\| _{C^0\left( B_{4/5} \right)}=0
\end{equation}
Hence we know $u^*$ is the weak solution of the equation $\mathrm{div}\left( A^*\nabla u \right) =0 \ in\,\,B_{3/4}$, and by our assumption, we know
\begin{equation}
    u^*\left( 0 \right) =0, \underset{B_{2/5}}{\sup}|u^*|=\underset{B_{3/4}}{\sup}|u^*|=1
\end{equation}
This contradicts to the maximum principle.
\end{proof}

With above lemmas, we can prove the weakened version of Theorem \ref{general conclusion}. We will later use this results to prove the stronger version:
\begin{theorem}
    Assume $A\in {A} \left( \lambda ,M \right) $, $\mathrm{div}\left( A\nabla u \right) =0  \ in\ B_1$ ,$u(0)=0$ and $u$ is not zero. Then there exists a constant $C$ depending only on $\lambda,M$ such that 
\begin{equation}
    \mathcal{H}^1( \left\{ u=0 \right\} \cap B_1 ) \geq C
\end{equation}
\end{theorem}
\begin{proof}
Let $N:=\log _2\small{\frac{\underset{B_{4/5}}{\sup}|u|}{\underset{B_{2/5}}{\sup}|u|}}$. By Lemma \ref{doubling lower} , we can assume
\begin{equation}
    N\geq C_1=C_1\left( \lambda ,M \right)>0 
\end{equation}
Divide the interval $[\frac{1}{2},\frac{7}{10}]$ equally into
\begin{equation}
    \small{\frac{1}{2}=r_0<r_1<...<r_k=\small{\frac{7}{10}}}
\end{equation}
and define
\begin{equation}
    M_i=\underset{B_{r_i}}{\sup}u, \ m_i=\underset{B_{r_i}}{\inf}u,\  i=0,1,2...k
\end{equation}
WLOG, we assume that $\underset{B_{2/5}}{\sup}|u|=1$. Then by Lemma \ref{Harnack type} , we can assume
\begin{equation}
    M_i\geq C_{2}(\lambda), \ m_i\leq-C_{2}(\lambda),\ i=0,1,2,...k
\end{equation}
For a fix constant $s>1$, we define
\begin{equation}
    I_{s}^{+}:=\left\{ i, M_i\leq sM_{i+1} \right\} , I_{s}^{-}:=\left\{ i, -m_i\leq -sm_{i+1} \right\} 
\end{equation}
\begin{equation}
    k^+:=\#I_{s}^{+}, k^-:=\#I_{s}^{-}
\end{equation}
Then by Lemma \ref{Harnack type}
\begin{equation}
    s^{k-k^+}\leq \small{\frac{\underset{B_{7/10}}{\sup}u}{\underset{B_{1/2}}{\sup}u}}\leq \small{\frac{\underset{B_{4/5}}{\sup}|u|}{C_2\left( \lambda \right)}\leq 2^{C_3\left( \lambda ,M \right) N}}
\end{equation}
Hence, we obtain
\begin{equation}
    k^+\geq k-\left( \log _s2 \right) C_3\left( \lambda ,M \right) N
\end{equation}
Hence by Lemma \ref{doubling lower}, there exists a constant $s_0=s_0(\lambda,M)$ such that if $s\leq s_0$, then
\begin{equation}
    \left( \log _s2 \right) C_3\left( \lambda ,M \right) N>1000
\end{equation}
Take $s=s_0$ and $k=3\left( \log _{s_0}2 \right) C_3\left( \lambda ,M \right) N$. Then, we have the estimate of $k^+$
\begin{equation}
    k^+\geq \small{\frac{2}{3}k}
\end{equation}
With a similar argument, we also have the estimate of $k^-$
\begin{equation}
    k^-\geq \small{\frac{2}{3}k}
\end{equation}
Hence, the following inequality holds
\begin{equation}
    \#\left( I_{s}^{+}\cap I_{s}^{-} \right) \geq \small{\frac{1}{3}k}
\end{equation}
For any $i\in I_{s}^{+}\cap I_{s}^{-}$, we take $x_i,y_i\in \partial B_{r_i}$ such that
\begin{equation}
    u\left( x_i \right) =\underset{B_{r_i}}{\sup}\ u,\  u\left( y_i \right) =\underset{B_{r_i}}{\inf}\ u
\end{equation}
By scaling, Schauder estimate and Lemma \ref{doubling lower}, we know
\begin{equation}
    \underset{B\left( x_i,\small{\frac{C_4\left( \lambda ,M \right)}{N}} \right)}{\sup}|\nabla u|\leq C_5\left( \lambda ,M \right) \underset{B\left( x_i,\small{2\frac{C_4\left( \lambda ,M \right)}{N}} \right)}{\sup}|u|\leq C_6\left( \lambda ,M \right) \underset{B_{r_{i+1}}}{\sup}u\leq C_7\left( \lambda ,M \right) \underset{B_{r_i}}{\sup} u
\end{equation}
Thus, we can estimate $u$ as below
\begin{equation}
    u>0 \ in\,\,B\left( x_i,\small{\frac{C_8\left( \lambda ,M \right)}{N}} \right) , B\left( x_i,\small{\frac{C_8\left( \lambda ,M \right)}{N}} \right) \subset B_{r_{i+1}}\verb|\|B_{r_{i-1}}
\end{equation}
Also we know that
\begin{equation}
    u<0 \ in\,\,B\left( y_i,\small{\frac{C_8\left( \lambda ,M \right)}{N}} \right) , B\left( y_i,\small{\frac{C_8\left( \lambda ,M \right)}{N}} \right) \subset B_{r_{i+1}}\verb|\|B_{r_{i-1}}
\end{equation}
By the continuity of $u$, we know that
\begin{equation}
    \mathcal{H}^1(\left\{ u=0 \right\} \cap B_{r_{i+1}}\verb|\|B_{r_{i}} )\geq \frac{C_8\left( \lambda ,M \right)}{N}
\end{equation}
Finally, we have the estimate
\begin{equation}
    \mathcal{H}^1( \left\{ u=0 \right\} \cap B_{\frac{7}{10}}\verb|\|B_{\frac{1}{2}} ) \geq \frac{C_8\left( \lambda ,M \right)}{N}\small{\frac{k}{3}} \geq C_9\left( \lambda ,M \right)
\end{equation}
\end{proof}

Finally, we can prove the general conclusion:

\begin{proof}[Proof of Theorem \ref{general conclusion}]
WLOG, we assume that $Lip\left( a_{ij} \right) <\infty $, otherwise we can consider the equation on $B_{1/2}$. Assume $Lip\left( a_{ij} \right) <M $, $M$ is an interger and $\ M>1000 $, then we can blow up the equation as
\begin{equation}
    v\left( x \right) =u\left( \small{\frac{x}{M}} \right) , b_{ij}=a_{ij}\left( \small{\frac{x}{M}} \right) ,B=\left[ b_{ij} \right] 
\end{equation}
Then, we just need to consider the following equation
\begin{equation}
    \mathrm{div}\left( B\nabla v \right) =0 \ in \ B_M, B\in A\left( \lambda ,1 \right)
\end{equation}
By maximum principle and mean value theorem we can assume
\begin{equation}
    z_i\in \partial B_i, v\left( z_i \right) =0, i=1,2,3,...,M
\end{equation}
Then by Lemma \ref{doubling lower}, we have the estimate
\begin{equation}
   \mathcal{H}^1(\left\{ v=0 \right\} \cap B\left( z_i,1/2 \right) ) \geq C\left( \lambda \right) 
\end{equation}
Add them up, we can get
\begin{equation}
    \mathcal{H}^1( \left\{ v=0 \right\} \cap B_M ) \geq C\left( \lambda \right) M
\end{equation}
After scaling, we can get our final result as below
\begin{equation}
    \mathcal{H}^1( \left\{ u=0 \right\} \cap B_1 )\geq C\left( \lambda \right)
\end{equation}
\end{proof}

Also, by similar argument of Lemma \ref{Main Lemma}, we can prove the following theorem:
\begin{theorem}
    Assume that $A\in {A} \left( \lambda ,M \right) $, $\mathrm{div}\left( A\nabla u_0 \right) =0  \ in\ B_1$ ,$u_0(0)=0$ and $u_0$ is not zero. The function $u$ is continuous. There exists two constant $\varepsilon,C>0$ depending only on $\lambda,M$ such that if 
    \begin{equation}
        \sup_{B(0,1)}|u-u_{0}|\leq \varepsilon\sup_{B(0,1/2)}|u_{0}|
    \end{equation}
    then we have 
    \begin{equation}
        \mathcal{H}^{1}(\{u=0\}\cap B(0,1))\geq C
    \end{equation}
\end{theorem}
It shows that the nodal volume still has lower bound under small perturbation. But the quantitative lower bound is still unknow.

\section{Special Case in SMP function}

In this section, we will derive the lower bound in a different way and all the functions are defined in $\mathbb{R}^2$. It only needs the basic definition of Hausdorff measure. Also, it only requires the function to be continuous and has strong maximum principle. We also hope this argument will work for higher dimension. For convenience, we will use the following notation:

\begin{definition}
    Let
    \begin{equation}
    SMP\left( B_1 \right) :=\left\{ u\in C^0\left( \overline{B_1} \right) :\text{for    }\forall \varOmega \subseteq B_1,\varOmega \text{ open},u|_{\varOmega}
\text{ cannot take maximum or minimum } in\,\,\varOmega \right\}   
    \end{equation}
    And we say $u$ is SMP function if $u\in SMP\left( B_1 \right) $.
\end{definition}

\begin{proof}[Proof of Theorem \ref{main results}]
    If we take $u(x_1,x_2)=x_1$, we know that
    \begin{equation}
        \underset{\begin{array}{c}
	u\in SMP\left( B_1 \right)\\
	u\left( 0 \right) =0\\
\end{array}}{inf}\mathcal{S}^1( \left\{ u=0 \right\} \cap \overline{B_1} )\leq 2
    \end{equation}
We only need to prove that if $u\in SMP\left( B_1 \right) ,u\left( 0 \right) =0$, then
\begin{equation}
    \mathcal{S}^1(\left\{ u=0 \right\} \cap \overline{B_1} ) \geq 2
\end{equation}
To prove this, we introduce the following set
\begin{equation}
    I:=\left\{ s\in \left[ 0,1 \right] ,\mathcal{S}^1( \left\{ u=0 \right\} \cap T_{s,1} )\geq 2\left( 1-s \right) \right\} 
\end{equation}
where $T_{s,1}=\left\{ s\leq|x|\leq1 \right\} $. WLOG, we assume that $\mathcal{S}^1 (\left\{ u=0 \right\} \cap \overline{B_1} ) <\infty $. Otherwise, there is nothing to prove. We will prove the final result with the following two claims.

Claim 1: $I\ne \emptyset$ and $I$ is closed.
\begin{proof}[Proof of Claim 1]
It is obvious that $1\in I$, so $I\ne \emptyset $. Take any sequence $\left\{ s_n \right\} _{n=1}\in I$ and$\underset{n\rightarrow \infty}{\lim}s_n=s$. Then there are two cases, the first case is that there exists a subsequence $\left\{ s_{n_k} \right\} _{k=1}$ monotonically increasing and converging to $s$. The second case is that there exists a subsequence $\left\{ s_{n_k} \right\} _{k=1}$ monotonically decreasing and converging to $s$.

For the first case ,by $\mathcal{S}^1( \left\{ u=0 \right\} \cap \overline{B_1} ) <\infty $, we know that
\begin{equation}
    \mathcal{S}^1(\left\{ u=0 \right\} \cap T_{s,1} ) =\underset{k\rightarrow \infty}{\lim}\mathcal{S}^1(\left\{ u=0 \right\} \cap T_{s_{n_k},1} ) \geq 2\left( 1-s \right) 
\end{equation}
For the second case, similarly, we have
\begin{equation}
     \mathcal{S}^1( \left\{ u=0 \right\} \cap T_{s,1} ) \geq \mathcal{S}^1(\left\{ u=0 \right\} \cap T_{s_{n_k},1} ) \geq 2\left( 1-s_{n_k} \right) 
\end{equation}
Thus, we obtain
\begin{equation}
    \mathcal{S}^1(\left\{ u=0 \right\} \cap T_{s,1} ) \geq \underset{k\rightarrow \infty}{\lim}2\left( 1-s_{n_k} \right)=2(1-s) 
\end{equation}
This finishes the proof of Claim 1.
\end{proof}

Claim 2: If $s\in I$ and $s>0$, then there exists a constant $\delta_s>0$ such that 
\begin{equation}
    [s-\delta_s,s]\subset I
\end{equation}
\begin{proof}[Proof of Claim 2]
By definition of SMP function, we can assume
\begin{equation}
    x_s,y_s\in \partial B_s\,\,u\left( x_s \right) =\underset{B_s}{\max}u>0,u\left( y_s \right) =\underset{B_s}{\min}u<0
\end{equation}
By continuity of $u$, we can assume that there exists $r_s>0$ such that  
\begin{equation}
    \,\,\begin{cases}
	u\left( x \right) >0,x\in B\left( x_s,r_s \right)\\
	u\left( x \right) <0,x\in B\left( y_s,r_s \right)\\
\end{cases}
\end{equation}
We will give the proof in polar coordinates. Let $x_s=\left( s,\theta _1 \right) ,y_s=\left( s,\theta _2 \right),0\leq \theta_1<\theta_2<2\pi $. We will consider two sets as below
\begin{equation}
   \begin{cases}
	Z_1:= \left\{ u=0 \right\} \cap \left\{ s-\frac{r_s}{2}\leq r<s,\theta _1\leq \theta \leq \theta _2 \right\} \\
	Z_2:= \left\{ u=0 \right\} \cap \left\{ s-\frac{r_s}{2}\leq r<s,\theta _2\leq \theta \leq 2\pi \,\,or\,\,0\leq \theta \leq \theta _1 \right\} \\
\end{cases} 
\end{equation}
And we consider the following two line segment 
\begin{equation}
    \begin{cases}
	l_1:=\left\{ s-\frac{r_s}{2}\leq r<s,\theta =\frac{\theta _1+\theta _2}{2} \right\}\\
	l_2:=\left\{ s-\frac{r_s}{2}\leq r<s,\theta =\frac{\theta _1+\theta _2}{2}+\pi \right\}\\
\end{cases}
\end{equation}
We consider any sets of balls $\left\{ B_{i}^{1}=B\left( x_{i}^{1},r_{i}^{1} \right) \right\} _{i=1}^{\infty},\left\{ B_{i}^{2}=B\left( x_{i}^{2},r_{i}^{2} \right) \right\} _{i=1}^{\infty}$ such that 
\begin{equation}
    \begin{cases}
	Z_1\subseteq \underset{i=1}{\overset{\infty}{\cup}}B_{i}^{1},Z_1\cap B_{i}^{1}\ne \emptyset ,r_{i}^{1}\leq \frac{r_s}{10},\forall i=1,2,...\\
	Z_2\subseteq \underset{i=1}{\overset{\infty}{\cup}}B_{i}^{2},Z_2\cap B_{i}^{2}\ne \emptyset ,r_{i}^{2}\leq \frac{r_s}{10},\forall i=1,2,...\\
\end{cases}
\end{equation}
Then by our construction, we know that
\begin{equation}
    B_{i}^{1}\cap B_{j}^{2}=\emptyset ,\forall i,j=1,2,...
\end{equation}
We can rotate $\{x_{i}^{1}\}$ along the circle $\left\{ r=|x_{i}^{1}| \right\}$ and also rotate $\{x_{i}^{2}\}$ along the circle $\left\{ r=|x_{i}^{2}| \right\}$ to the points
\begin{equation}
    \begin{cases}
	\widetilde{x_{i}^{1}}=\left( |x_{i}^{1}|,\small{\frac{\theta _1+\theta _2}{2}} \right)\\
	\widetilde{x_{i}^{2}}=\left( |x_{i}^{2}|,\small{\frac{\theta _1+\theta _2}{2}+\pi} \right)\\
\end{cases}
\end{equation}
Then by the property of continuous function, we know that for any $s-\small{\frac{r_s}{2}\leq t<s}$
\begin{equation}
    \left\{ u=0 \right\} \cap \left\{ r=t,\theta _1<\theta <\theta _2 \right\} \ne \emptyset
\end{equation}
If we assume that $z_t\in \left\{ u=0 \right\} \cap \left\{ r=t,\theta _1<\theta <\theta _2 \right\} $, then we can find $i$ such that
\begin{equation}
    z_t\in B_{i}^{1}
\end{equation}
So we know that
\begin{equation}
    \left( t,\small{\frac{\theta _1+\theta _2}{2}} \right) \in \widetilde{B_{i}^{1}}
\end{equation}
Hence, we have the following relation
\begin{equation}
    l_1\subseteq \underset{i=1}{\overset{\infty}{\cup}}B\left( \widetilde{x_{i}^{1}},r_{i}^{1} \right)
\end{equation}
Similarly, we get
\begin{equation}
	l_2\subseteq \underset{i=1}{\overset{\infty}{\cup}}B\left( \widetilde{x_{i}^{2}},r_{i}^{2} \right)
\end{equation}
By the defnition of Spherical measure, we know that
\begin{equation}
    \mathcal{S}^1(\left\{ u=0 \right\} \cap \left\{ s-\small{\frac{r_s}{2}}\leq r<s \right\} ) \geq \mathcal{S}^1( l_1\cup l_2 ) =r_s
\end{equation}
Therefore, we obtain
\begin{equation}
    \mathcal{S}^1( \left\{ u=0 \right\} \cap \left\{ s-\small{\frac{r_s}{2}}\leq r \leq1\right\} ) \geq 2(1-(s-\frac{r_s}{2}))
\end{equation}
This finishes the proof of Claim 2.
\end{proof}

By Claim 1 and Claim 2, we can prove $I=[0,1]$. Firstly, by Claim 2 we know that the interval $(0,1)\verb|\|I$ is open if $(0,1)\verb|\|I\ne \emptyset$. Then we assume that
\begin{equation}
    \left( 0,1 \right) \verb|\|I=\underset{i=1}{\overset{\infty}{\cup}}\left( a_i,b_i \right) 
\end{equation}
And we assume that $b_1>b_2>\cdots$. Then $b_1\in I$ since $1\in I$. By Claim 2, we can find $\delta _{b_i}>0$ such that
\begin{equation}
    \left[ b_i-\delta _{b_i},b_i \right] \subseteq I
\end{equation}
This leads to contradiction! So we know that
\begin{equation}
    \left( 0,1 \right) \subseteq I
\end{equation}
At the same time, by Claim 1, we know that
\begin{equation}
   I=[0,1]
\end{equation}
This ends the proof.
\end{proof}

\bibliographystyle{alpha}
\bibliography{reference}
{\small 
\indent (Jiahuan Li) SCHOOL OF MATHEMATICS SCIENCE, UNIVERSITY OF SCIENCE AND TECHNOLOGY OF CHINA, HEFEI, 230022, CHINA.\;
Email.address: jiahuan@mail.ustc.edu.cn\\ \\
(Zhichen Ying) SCHOOL OF MATHEMATICAL SCIENCES, ZHEJIANG UNIVERSITY, HANGZHOU, 310058, CHINA.\;
Email address: yingzc059@gmail.com
}
\end{document}